\documentclass{amsart}

\usepackage{graphicx}
\usepackage{mathrsfs} % For \pg and \rpg
\usepackage{url}
\usepackage[tight]{subfigure}

\newcommand{\false}{\textup{\texttt{false}}}
\newcommand{\mfd}{\mathcal{M}}
\newcommand{\normaliz}{Normaliz}
\newcommand{\python}{Python}
\newcommand{\recogniser}{3-Manifold Recogniser}
\newcommand{\rpg}{\mathscr{P}_1(\mfd)}
\newcommand{\regina}{Regina}
\newcommand{\snappea}{SnapPea}
\newcommand{\snappy}{SnapPy}
\newcommand{\tri}{\mathcal{T}}
\newcommand{\true}{\textup{\texttt{true}}}

\newcommand{\R}{\mathbb{R}}
\newcommand{\Z}{\mathbb{Z}}

\newcommand{\menub}[2]{\textit{#1}\,$\longrightarrow$\,\textit{#2}}

\newcommand{\cpp}{C%
    \nolinebreak\hspace{-.05em}\raisebox{.4ex}{\tiny\bf +}%
    \nolinebreak\hspace{-.10em}\raisebox{.4ex}{\tiny\bf +}}

\newtheorem{theorem}{Theorem}[section]
\newtheorem{lemma}[theorem]{Lemma}
\newtheorem{algorithm}[theorem]{Algorithm}
\theoremstyle{definition}
\newtheorem{defn}[theorem]{Definition}

\begin{document}

\title[Computational topology with Regina]
    {Computational topology with Regina: \\
     Algorithms, heuristics and implementations}
\author{Benjamin A.\ Burton}
\address{School of Mathematics and Physics \\
    The University of Queensland \\
    Brisbane QLD 4072 \\
    Australia}
\email{bab@maths.uq.edu.au}
\thanks{The author is supported by the Australian Research Council
    under the Discovery Projects funding scheme (projects DP1094516
    and DP110101104).}
\dedicatory{Dedicated to J.\ Hyam Rubinstein}
 % Triangulating manifolds

\begin{abstract}
    Regina is a software package for studying 3-manifold triangulations
    and normal surfaces.  It includes a graphical user interface and
    Python bindings, and also supports angle structures, census
    enumeration, combinatorial recognition of triangulations, and
    high-level functions such as 3-sphere recognition, unknot
    recognition and connected sum decomposition.

    This paper brings 3-manifold topologists
    up-to-date with \emph{Regina} as it appears today,
    and documents for the first time in the literature
    some of the key algorithms, heuristics and implementations that are
    central to \emph{Regina}'s performance.
    These include the all-important simplification heuristics,
    key choices of data structures and algorithms to
    alleviate bottlenecks in normal surface enumeration, modern
    implementations of 3-sphere recognition and connected sum decomposition,
    and more.
    We also give some historical background for the project,
    including the key role played by Rubinstein in its genesis 15 years ago,
    and discuss current directions for future development.
\end{abstract}

\maketitle

%------------------------------ Introduction -----------------------------

\section{Introduction}

Algorithmic problems run to the heart of low-dimensional topology.
Prominent amongst these are \emph{decision problems}
(e.g., recognising the unknot, or testing whether two triangulated
manifolds are homeomorphic);
\emph{decomposition problems} (e.g., decomposing a triangulated manifold
into a connected sum of prime manifolds);
and \emph{recognition problems} (e.g., ``naming'' the manifold described by
a given triangulation).

For 3-manifolds in particular, such problems typically have highly complex and
expensive solutions---running times are often exponential or
super-exponential, and implementations are often major endeavours
developed over many years (if they exist at all).
This is in contrast to dimension two, where many such problems are
easily solved in small polynomial time, and dimensions $\geq 4$,
where such problems can become provably undecidable
\cite{chiodo11-groups,markov60-insolubility}.

{\regina} \cite{burton04-regina,regina} is a software package for
low-dimensional topologists and knot theorists.  Its main focus is on
triangulated 3-manifolds, though it is now branching into both two and
four dimensions also (see Section~\ref{s-future}).
It provides powerful algorithms and heuristics to assist
with decision, decomposition and recognition problems; more broadly,
it includes a range of facilities to study and
manipulate 3-manifold triangulations.
It offers rich support for normal surface theory, a major
algorithmic framework that recurs throughout 3-manifold topology.

One of {\regina}'s most important uses is for \emph{experimentation}:
it can analyse a single example too large to
study by hand, or millions of examples ``in bulk''.
Other uses include \emph{computer proofs},
such as the recent resolution of Thurston's old question
of whether the Weber-Seifert dodecahedral space is non-Haken \cite{burton12-ws},
and \emph{calculation} of previously-unknown topological invariants,
such as recent improvements to crosscap numbers in knot tables
\cite{burton11-hilbert,burton12-crosscap}.

The remainder of this introduction gives a brief overview of {\regina} and
some history behind its early development, back to its genesis with
Rubinstein 15 years ago.  Following this, we select some of {\regina}'s
most important features and study the modern heuristics, algorithms
and implementations behind them.  Some of these details are crucial for
{\regina}'s fast performance, but have not appeared in the literature to date.
Overall, this paper offers a fresh update on the significant progress
and performance enhancements in the eight years since {\regina}'s last
major report in the literature \cite{burton04-regina}.

The specific features that we focus on are:
\begin{itemize}
\item
the process of simplifying triangulations, which plays a critical role in many
of {\regina}'s high-level algorithms;
\item
the enumeration of normal and almost normal
surfaces, including the high-level 0-efficiency, 3-sphere recognition
and connected sum decomposition algorithms;
\item
combinatorial recognition, whereby {\regina} attempts to recognise a
triangulation by studying its combinatorial structure;
\item
the enumeration of angle structures, taut angle structures and veering
structures on triangulations.
\end{itemize}

With respect to these features, there are interesting new
developments in the pipeline.
Examples include ``breadth-first simplification'' of
triangulations, the tree traversal algorithm for enumerating
vertex normal surfaces, and new 0-efficiency and 3-sphere recognition
algorithms based on linear programming that exhibit experimental polynomial-time
behaviour.
The corresponding code is already written
(it is now being tested and prepared
for integration into {\regina}),
and we outline these new developments in the relevant
sections of this paper.

To finish, Section~\ref{s-exp} gives concrete examples
of how {\regina} can be used for experimentation,
including an introduction to {\regina}'s in-built {\python} scripting,
so that interested
readers can use these as launching points for their own experiments.
Section~\ref{s-future} concludes with
additional future directions for {\regina},
including new work with Ryan Budney on triangulated \emph{4-manifolds},
much of which is already coded and running in the development
repository.

\subsection{Overview of Regina}

{\regina} is now 13 years old, with over $190\,000$ lines of source code.
% 280534 lines in total (cc, cpp, c, h, tcc, hpp, py)
%    754 from rel (auto-built)
%   8249 from normaliz
%  76733 from snappea
%    178 from snappy
%   1497 from srchiliteqt
%
% 193123 native to regina
It is released under the GNU General Public License, and contributions
from the research community are welcome.
It adheres to the following broad development principles, in order of
precedence:
\begin{enumerate}
    \item \emph{Correctness:} Having correct output is critical,
    particularly for applications such as computer proofs.
    {\regina} is extremely conservative in this regard:
    for example, it will use arbitrary precision integer arithmetic
    if it cannot be proven unnecessary, and the API documentation
    makes thorough use of preconditions and postconditions.

    \item \emph{Generality:} Algorithms operate in the
    broadest possible scenarios (within reason), and do not
    require preconditions that cannot be easily tested.
    For instance, unknot recognition runs correctly for both bounded and
    ideal triangulations, and even when the input triangulation
    is not known to be a knot complement
    (whereupon it generalises to solid torus recognition).

    \item \emph{Speed:} Because many of its algorithms run in
    exponential or super-ex\-po\-nen\-tial time, speed is imperative.
    {\regina} makes use of sophisticated algorithms and heuristics that,
    whilst adhering to the constraints of correctness and generality,
    make it practical for real topological problems.
\end{enumerate}

Regina is multi-platform, and offers a drag-and-drop installer for MacOS,
an MSI-based installer for Windows, and ready-made packages for several
GNU/Linux distributions.
It is thoroughly documented, and stores its data files in a
compressed XML format.  Regina provides three levels of user interface:
\begin{itemize}
    \item a full graphical user interface, based on the Qt framework
    \cite{qt};
    \item a scripting interface based on {\python}, which can interact
    with the graphical interface or be used a stand\-alone {\python} module;
    \item a programmers' interface offering native access to
    {\regina}'s mathematical core through a {\cpp} shared library.
\end{itemize}

There are facilities to help new users learn their way around,
including an illustrated users' handbook, context-sensitive
``what's this?''\ help, and sample data files that can be opened through
the \menub{File}{Open Example} menu.

{\regina}'s core strengths are in working with triangulations, normal surfaces
and angle structures.  It only offers basic support for hyperbolic geometries,
for which the software packages {\snappea}~\cite{snappea} and its successor
{\snappy}~\cite{snappy} are more suitable.
In addition to its range of low-level operations and 3-manifold invariants,
{\regina} includes implementations of high-level decision and decomposition
algorithms, including the only known full implementations of
3-sphere recognition and connected sum decomposition.
For a comprehensive list of features, see
\url{http://regina.sourceforge.net/docs/featureset.html}.

\subsection{History}

In the 1990s, Jeff Weeks' software package {\snappea} had opened up enormous
possibilities for computer experimentation with hyperbolic 3-mani\-folds.
In contrast, outside the hyperbolic world many fundamental 3-manifold
algorithms remained purely theoretical, including
high-profile algorithms such as 3-sphere recognition, connected sum
decomposition, Haken's unknot recognition algorithm, and testing for
incompressible surfaces.

More broadly, \emph{normal surface theory}---a ubiquitous technology in
algorithmic 3-manifold topology---had no concrete software implementation
(or at least none that had been publicised).  Many of the algorithms
that used normal surfaces
were so complex that it was believed that any attempt at
implementation would be both painstaking to develop and impractically slow
to run.

Nevertheless, Rubinstein sought to move these algorithms from the
theoretical realm to the practical, and in 1997 he assembled a team for
a new software project that would make \emph{computational}
normal surface theory a reality.
This initial team consisted of David Letscher, Richard Rannard, and myself.
Much planning was done and a little preliminary code was written, but then
the team members became occupied with different projects.

Letscher later resumed the project on his own,
and at a 1999 workshop on computational 3-manifold topology at
Oklahoma State University he presented \emph{Normal},
a Java-based application that could modify and simplify
triangulations, and enumerate vertex normal surfaces.
This was an exciting development, but the software was proof-of-concept
only, and later that month Letscher and I sat down to begin
again from scratch.  The successor would be a fast and robust software
package with a {\cpp} engine, would be designed for extensibility and
portability, and would support community-contributed ``add-ons''.

Letscher stayed with the project for the first year, and then moved on
to pursue different endeavours.  I remained with the project (a PhD
student at the time, with more time on my hands), and
the project---now known as \emph{Regina}---saw its first public release in
December~2000.
Development at this stage was taking place at Oklahoma State University,
where Bus Jaco was both a strong supporter and a significant influence
on the project in its formative years.

Around that time, a significant phase shift was taking place in normal
surface theory: Jaco and Rubinstein were developing their theory of
\emph{0-efficient triangulations}.
Their landmark 2003 paper
\cite{jaco03-0-efficiency} showed that key
problems such as 3-sphere recognition and connected sum decomposition
could be solved without the expensive and highly problematic operation of
\emph{cutting} along a normal surface, and that instead this could be
replaced by a \emph{crushing} operation that always reduces the number
of tetrahedra.  For the first time, the key algorithms of
3-sphere recognition and connected sum decomposition became both
feasible and practical to implement.
Shortly after, in March 2004, {\regina} became the first software package
to implement robust, conclusive solutions to these fundamental problems.

{\regina} has come a long way since development began in 1999.
It continues to grow in its mathematics, performance, usability,
and technical infrastructure, and it enjoys
contributions from a wide range of people (as noted in the following section).
For a list of highlights over the years, the reader is invited to peruse
the ``abbreviated changelog'' at
\url{http://regina.sourceforge.net/docs/history.html}.

\subsection{Acknowledgements}

{\regina} is a mature project and has benefited from the expertise of
many people and the resources of many organisations.

Ryan Budney and William Pettersson have both made significant and
ongoing contributions, and are both now primary developers for the project.
A host of other people have contributed either code or expertise,
including Bernard Blackham, Marc Culler, Dominique Devriese, Nathan Dunfield,
Matthias Goerner, William Jaco, David Letscher, Craig Macintyre, Melih Ozlen,
Hyam Rubinstein, Jonathan Shewchuk, and Stephan Tillmann.
In addition, the open-source software libraries {\normaliz}
\cite{bruns10-normaliz} (by Winfried Bruns, Bogdan Ichim and Christof Soeger)
and {\snappea} \cite{snappea} (by Jeff Weeks) have been grafted directly into
{\regina}'s calculation engine.

Organisations that have contributed to the development of {\regina} include
the American Institute of Mathematics,
the Australian Research Council (Discovery Projects DP0208490, DP1094516
and DP110101104),
the Institute for the Physics and Mathematics of the Universe
(University of Tokyo),
Oklahoma State University,
the Queensland Cyber Infrastructure Foundation,
RMIT University,
the University of Melbourne,
% Note to copyeditors: apparently the correct way to spell
% The University of Queensland is always with a capital T, regardless of
% its position in the sentence.  They get very tetchy about this.
The University of Queensland,
the University of Victoria (Canada),
and the Victorian Partnership for Advanced Computing (Australia).

%------------------------------ Simplification ---------------------------

\section{Simplifying triangulations} \label{s-simp}

The most basic object in {\regina} is a \emph{3-manifold triangulation}.
{\regina} does not restrict itself to simplicial complexes; instead it
uses \emph{generalised triangulations}, a broader notion that
can represent a rich array of 3-manifolds using very few tetrahedra,
as described in Section~\ref{s-simp-tri} below.

One of {\regina}'s most important functions is \emph{simplification}:
retriangulating a 3-manifold to use very few tetrahedra (ideally, the
fewest possible).
{\regina} includes a rich array of local simplification moves,
as outlined in Section~\ref{s-simp-moves}; together it combines
these moves into a strong simplification algorithm, which we describe in
full in Section~\ref{s-simp-alg}.  Moreover, this algorithm is fast:
in Section~\ref{s-simp-imp} we outline some important implementation
details, and show that it runs in small polynomial time.

This simplification algorithm gives no guarantee of achieving a
\emph{minimal} triangulation (where the number of tetrahedra is
the smallest possible).
However, the algorithm is extremely powerful in practice:
for instance, of the $652\,635\,906$ combinatorially distinct triangulations
of the 3-sphere with $3 \leq n \leq 10$ tetrahedra \cite{burton11-genus},
this simplification algorithm quickly reduces
\emph{all but 26} of them.

Having a powerful simplification algorithm is essential in computational
3-ma\-ni\-fold topology, for several reasons:
\begin{itemize}
    \item Many significant algorithms have \emph{running times that are
    exponential or worse} in the number of tetrahedra, and so reducing
    this number is of great practical importance.
    \item Good simplification may allow us to \emph{avoid these
    expensive algorithms entirely}.
    For instance, in 3-sphere recognition we could simplify the
    triangulation and then compare it to the known trivial
    ($\leq 2$-tetrahedron)
    triangulations of $S^3$: if they match then the expensive steps
    involving normal and almost normal surfaces can be avoided
    altogether \cite{burton11-orprime}.
    \item More generally, simplification offers good heuristics for
    solving the \emph{homeomorphism problem}.  Simplified triangulations
    are often well-structured, which makes {\regina}'s combinatorial
    recognition routines more likely to identify the underlying
    3-manifold (see Section~\ref{s-rec}).
    Moreover, if two triangulations $\tri$ and $\tri'$ represent
    the same 3-manifold, then after simplification it often becomes relatively
    easy to convert one into the other using Pachner moves
    (i.e., bistellar flips) \cite{burton11-orprime}.
\end{itemize}

Most of the individual local moves that we describe in
Section~\ref{s-simp-moves}
are well-known, and were implemented by Letscher in his early software
package \emph{Normal}.  Moreover,
Matveev \cite{matveev03-algms,matveev90-complexity}
and Martelli and Petronio \cite{martelli01-or9} perform similar moves in
the dual setting of special spines, and the author has described several
in the context of census enumeration \cite{burton04-facegraphs}.
We therefore outline these moves very briefly, although in our general
setting there are prerequisites for preserving the underlying 3-manifold
that earlier authors have not required (either because they worked in
the more flexible dual setting of spines, or they were simply proving
criteria for minimality and/or irreducibility).
We pay particular attention to the \emph{edge collapse} move,
whose prerequisites are complex and require non-trivial algorithms to
test quickly.
The combined simplification algorithm, whose underlying heuristics have
been fine-tuned over many years, is described here in
Section~\ref{s-simp-alg} for the first time.

In Section~\ref{s-simp-bfs}, we conclude our discussion on simplification
with a new technology that will soon make its way into {\regina}:
simplification by \emph{breadth-first search through the Pachner graph}.
Unlike the algorithm of Section~\ref{s-simp-alg}, this is not
poly\-no\-mial-time; however, it is found to be extremely effective in
handling the (very rare) pathological triangulations that cannot be
simplified otherwise.

\subsection{Generalised triangulations} \label{s-simp-tri}

A \emph{generalised 3-manifold triangulation}, or just
a \emph{triangulation} from here on, is formed from $n$ tetrahedra by
affinely identifying (or ``gluing'') some or all of their $4n$ faces in pairs.
A face is allowed to be identified with another face of the same tetrahedron.
It is possible that several edges of a single tetrahedron may be
identified together as a consequence of the face gluings,
and likewise for vertices.  It is common to work with
\emph{one-vertex triangulations}, in which all vertices of all tetrahedra
become identified as a single point.

Generalised triangulations can produce rich constructions with very few
tetrahedra.  For instance, Matveev obtains $103\,042$ distinct
closed orientable irreducible 3-manifolds from just
$\leq 13$ tetrahedra \cite{matveev12-owr}, including representatives
from all of Thurston's geometries except for $S^2 \times \R$
(which cannot yield a manifold of this type \cite{scott83-geometries}).

Figure~\ref{fig-rp3} illustrates a two-tetrahedron triangulation of
the real projective space $\R P^3$.  The two tetrahedra are labelled
0 and 1, and the four vertices of each tetrahedron are labelled 0, 1, 2
and 3.  Faces 012 and 013 of tetrahedron~0 are joined directly to faces
012 and 013 of tetrahedron~1, creating a solid ball; then faces
023 and 123 of tetrahedron~0 are joined to faces 132 and 032
of tetrahedron~1, effectively gluing the top of the ball to
the bottom of the ball with a $180^\circ$ twist.

\begin{figure}
    \centering
    \includegraphics[scale=0.9]{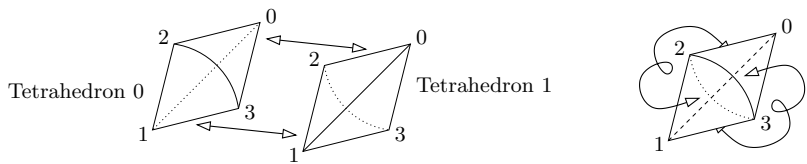}
    \caption{A triangulation of the real projective space $\R P^3$}
    \label{fig-rp3}
\end{figure}

All of this information can be encoded in a table of face gluings,
which is how {\regina} represents triangulations:

\medskip

\centerline{\small\begin{tabular}{c|r|r|r|r}
    Tetrahedron & Face 012 & Face 013 & Face 023 & Face 123 \\
    \hline
    0 & 1 (012) & 1 (013) & 1 (132) & 1 (032) \\
    1 & 0 (012) & 0 (013) & 0 (132) & 0 (032)
\end{tabular}}

\medskip

Consider the cell in the row for tetrahedron $t$ and the
column for face \textit{abc}.  If this cell
contains $u~(\mathit{xyz})$, this indicates that face
\textit{abc} of tetrahedron~$t$ is identified with face
\textit{xyz} of tetrahedron~$u$, using the affine gluing that maps
vertices $a$, $b$ and $c$ of tetrahedron~$t$ to vertices
$x$, $y$ and $z$ of tetrahedron~$u$ respectively.

For any vertex $V$ of a triangulation, the \emph{link} of $V$ is
defined as the frontier of a small regular neighbourhood of $V$.
%It is often
%useful to think of vertex links as triangulated 2-dimensional
%surfaces, formed by inserting a small triangle
%into each corner of each tetrahedron, and then joining
%together triangles from adjacent tetrahedra along their edges.
This mirrors the traditional concept of a link in a simplicial
complex, but is modified to support the generalised triangulations that
we use in {\regina}.

A \emph{closed triangulation} is one that represents a closed
3-manifold (like the example above): every tetrahedron face must
be glued to a partner, and every vertex link must be a 2-sphere.
A \emph{bounded triangulation} is one that represents a 3-manifold
with boundary: some tetrahedron faces are not glued to anything
(together these form the boundary of the manifold),
and every vertex link must be a 2-sphere or a disc.
In either case, it is important that no edge be identified
with itself in reverse as a consequence of the face gluings.

%The following table describes a one-tetrahedron
%triangulation of the solid torus $B^2 \times S^1$, whose boundary
%consists of two triangles (faces 023 and 123 of tetrahedron~0) that
%together form a 2-dimensional torus:
%
%\medskip
%
%\centerline{\small\begin{tabular}{c|r|r|r|r}
%    Tetrahedron & Face 012 & Face 013 & Face 023 & Face 123 \\
%    \hline
%    0 & 0 (301) & 0 (120) & &
%\end{tabular}}
%
%\medskip

{\regina} can also work with \emph{ideal} triangulations,
such as Thurston's famous two-tetrahedron triangulation of the
figure eight knot complement \cite{thurston78-lectures}: these are
triangulations in which vertex links can be
higher-genus closed surfaces.
%(such as tori or Klein bottles).
%Ideal triangulations can represent non-compact 3-manifolds
%(by deleting the vertices), or compact bounded 3-manifolds
%(by truncating the vertices).
For simplicity, in this paper we focus
our attention on \emph{closed and bounded triangulations only},
though most of our results apply equally well to ideal triangulations also.

Users can enter tetrahedron gluings directly into {\regina}, or
create triangulations in other ways: importing from
{\snappea} \cite{snappea} or other file formats;
building ``pre-packaged'' constructions
such as layered lens spaces or Seifert fibred spaces;
entering \emph{dehydration strings} \cite{callahan99-cuspedcensus}
or the more flexible \emph{isomorphism signatures} \cite{burton11-orprime}
(short pieces of text that completely encode a triangulation);
or accessing large ready-made
censuses that hold tens of thousands of triangulations of various types.

\subsection{Individual simplification moves} \label{s-simp-moves}

In this section we outline several individual local moves on
triangulations.  Following this, in Section~\ref{s-simp-alg} we piece
these moves together to build {\regina}'s full simplification algorithm.

We describe each move in full generality, as applied to either closed or
bounded triangulations, and without any restrictions such as orientability
or irreducibility.  For each move we give sufficient conditions under which
the move is ``safe'', i.e., does not change the underlying 3-manifold.

To keep the exposition as short as possible,
we simply state these conditions without proof.
However, the proofs are simple: the key idea is that, throughout the
intermediate stages of each move, we never crush
an edge, flatten a bigon or flatten a triangular pillow
whose bounding vertices, edges or faces respectively are
either identified or both in the boundary.
For detailed examples of these
types of arguments, see (for instance) the proof of Lemma~3.7
in \cite{burton07-nor10}.

\subsubsection{Pachner-type moves}

Our first moves are simple combinations of the well-known
\emph{Pachner moves} \cite{pachner91-moves},
also known as \emph{bistellar flips}.
All of these moves, when performed locally within some larger
triangulation, preserve the underlying 3-manifold with no special
preconditions required.

\begin{defn}
    Consider a triangular bipyramid: this can be triangulated with
    (i)~two distinct tetrahedra joined along an internal face, or with
    (ii)~three distinct tetrahedra
    joined along an internal degree three edge, as illustrated in
    Figure~\ref{fig-23}.  A \emph{2-3 Pachner move} replaces (i) with
    (ii), and a \emph{3-2 Pachner move} replaces (ii) with (i).

    \begin{figure}
    \centering
    \subfigure[The 2-3 and 3-2 Pachner moves]{\label{fig-23}%
        \includegraphics[scale=0.7]{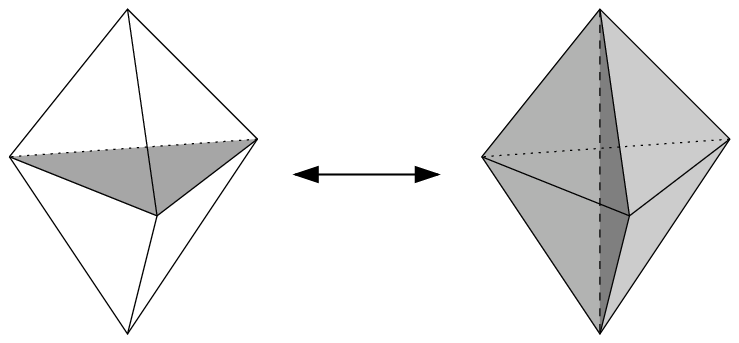}}
    \hspace{1cm}
    \subfigure[The ``aggregate'' 4-4 move]{\label{fig-44}%
        \includegraphics[scale=0.7]{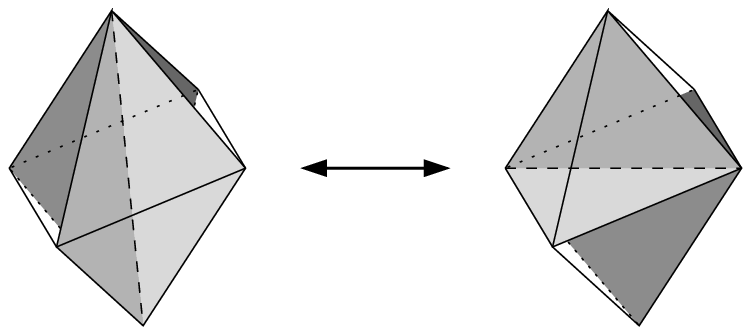}}
    \caption{Three Pachner-type moves}
    \end{figure}

    Consider a square bipyramid (i.e., an octahedron).
    This can be triangulated with four distinct tetrahedra joined along
    an internal degree four edge; moreover, there are three ways of
    doing this (since the internal edge could follow any of the three
    main diagonals of the octahedron).
    A \emph{4-4 move} replaces one such triangulation with another,
    as illustrated in Figure~\ref{fig-44}.
\end{defn}

There are two additional Pachner moves:
the \emph{1-4 move} and the \emph{4-1} move.
{\regina} does not use either of these: the 1-4 move
complicates the triangulation more than is necessary (since it
introduces a new vertex, which is never needed to simplify a
triangulation \cite{matveev87-transformations,matveev03-algms}), and the 4-1
move is a special case of an \emph{edge collapse} (described later in this
section).  The 4-4 move is not a Pachner move, but it can be expressed as
an ``aggregate'' of a 2-3 move followed by a 3-2 move.

\subsubsection{Moves around low-degree edges and vertices}

It is well-known that (under the right conditions) minimal triangulations
cannot have low-degree edges or vertices.  One typically proves this
using local moves around such edges or vertices that either reduce
the number of tetrahedra or break some prior assumption on the manifold.

Here we outline three such moves, which {\regina} uses in a more
general setting to simplify arbitrary triangulations.
Unlike the earlier Pachner-type moves,
these moves \emph{might} change the underlying 3-manifold; after describing the
moves, we list sufficient conditions under which they are ``safe''
(i.e., the 3-manifold is preserved).

\begin{defn}
    A \emph{2-0 vertex move} operates on a triangular pillow formed
    from two distinct tetrahedra surrounding an internal degree two vertex,
    as illustrated in Figure~\ref{fig-20v},
    and flattens this pillow to a single face.
    A \emph{2-0 edge move} operates on a bigon bipyramid formed from
    two distinct tetrahedra surrounding an internal degree two edge,
    as illustrated in Figure~\ref{fig-20e},
    and flattens this to a pair of faces.

    \begin{figure}
    \centering
    \subfigure[The 2-0 vertex move]{\label{fig-20v}%
        \includegraphics[scale=0.6]{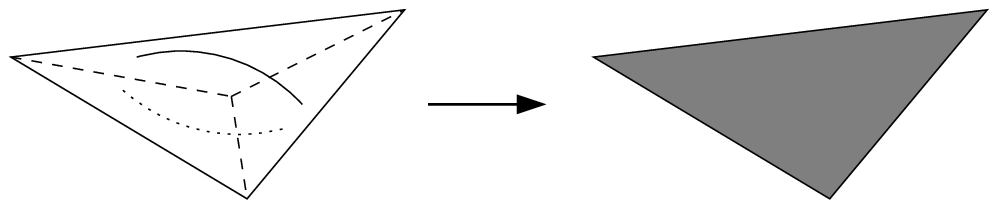}}
    \hspace{6.1cm}\phantom{x} % This is to make (a) and (b) line up vertically
    \\
    \subfigure[The 2-0 edge move]{\label{fig-20e}%
        \includegraphics[scale=0.6]{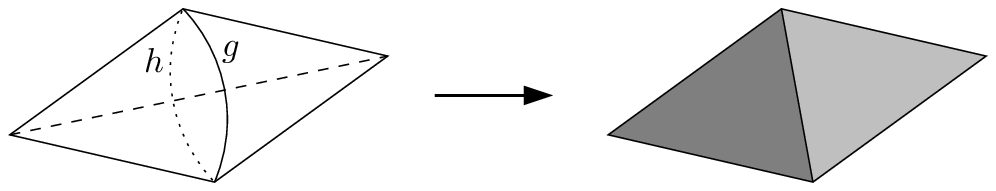}} 
    \hspace{0.8cm}
    \subfigure[The 2-1 edge move]{\label{fig-21e}%
        \smash{\includegraphics[scale=0.8]{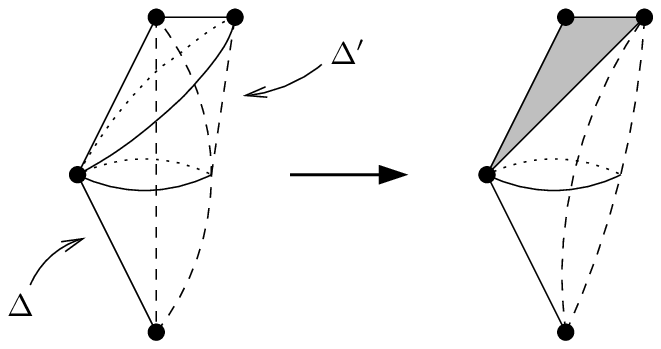}}}
    \caption{Local moves around low-degree edges and vertices}
    \end{figure}

    Consider a tetrahedron $\Delta$, two of whose faces are folded
    together around an internal degree one edge, and let $\Delta'$
    be some distinct adjacent tetrahedron as illustrated in
    Figure~\ref{fig-21e}
    (for clarity, the vertices of both tetrahedra are marked in bold).
    A \emph{2-1 edge move} flattens the two uppermost faces of $\Delta'$
    together and retriangulates the remaining region with a single
    tetrahedron to yield a new degree one edge, as shown in the illustration.
\end{defn}

To ensure that these moves do not change the underlying 3-manifold, the
following conditions are sufficient:
\begin{itemize}
    \item For the 2-0 vertex move, the two faces that bound the pillow
    must be distinct (i.e., not identified)
    and not both simultaneously in the boundary.

    \item For the 2-0 edge move, let $g$ and $h$ denote the two edges
    that we flatten together, as marked in Figure~\ref{fig-20e}.
    These edges must be distinct and not both boundary.
    Likewise, on each side of $g$ and $h$,
    the two faces that we flatten together must be distinct
    and not both boundary.
    Finally, although we may identify a face on one side of $g$ and $h$
    with a face on the other, we do not allow \emph{all four}
    faces to be identified in pairs, and we do not allow two of them
    to be identified if the other two are both in the boundary.

    \item For the 2-1 edge move, the two edges of $\Delta'$ that we
    flatten together must be distinct and not both boundary,
    and likewise the two faces of $\Delta'$ that we flatten together must
    be distinct and not both boundary.
    Note that the second constraint comes ``for free'' as a corollary of the
    first.
\end{itemize}

\subsubsection{Moves on the boundary}

Our next moves apply only to bounded triangulations.  As before, they
have the potential to change the underlying 3-manifold, and after
presenting the moves we give sufficient conditions under which they do not.

\begin{defn}
    An \emph{book opening move} operates on a face with precisely
    two of its three edges in the boundary of the triangulation,
    as illustrated in Figure~\ref{fig-book}, and ``unfolds'' the
    two tetrahedra on either side (which need not be distinct)
    to create two new boundary faces.
    A \emph{book closing move} is the inverse move: it operates
    on two distinct adjacent faces in the boundary and ``folds'' them
    together so that they become identified as a single internal face.

    \begin{figure}
    \centering
    \subfigure[The book opening and closing moves]{\label{fig-book}%
        \includegraphics[scale=0.8]{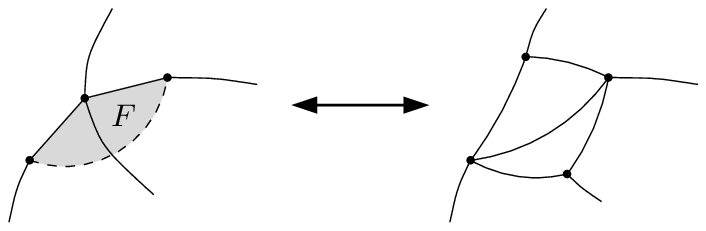}}
    \hspace{1cm}
    \subfigure[One type of boundary shelling move]{\label{fig-shell}%
        \includegraphics[scale=0.8]{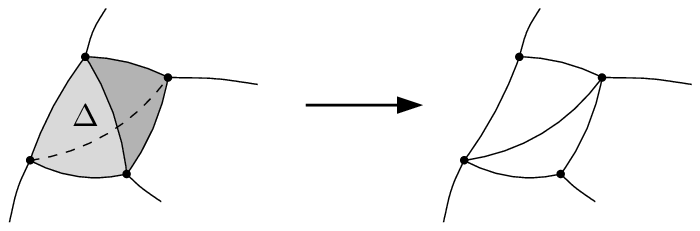}}
    \caption{Moves on the boundary of a 3-manifold triangulation}
    \end{figure}

    A \emph{boundary shelling move} operates on a tetrahedron $\Delta$
    that has precisely one, two or three faces in the boundary, and simply
    removes the tetrahedron from the triangulation (effectively ``plucking
    it off'' the boundary).  This move is illustrated in Figure~\ref{fig-shell}
    for the case where $\Delta$ has two faces in the boundary.
\end{defn}

The book opening move will always preserve the underlying 3-manifold.
The book closing move will preserve the 3-manifold if (i)~the
two boundary faces in question are not the \emph{only} faces in that
boundary component, and (ii)~the two vertices opposite the common edge
between these faces are not already identified.
The behaviour of the boundary shelling moves depends on the number of
faces of the tetrahedron $\Delta$ that lie in the boundary:
\begin{itemize}
    \item If $\Delta$ has three boundary faces, removing $\Delta$ always
    preserves the 3-manifold.
    \item If $\Delta$ has two boundary faces, we preserve
    the 3-manifold if (i)~the one edge of
    $\Delta$ not on these faces is internal to the triangulation,
    and (ii)~the two remaining (non-boundary)
    faces of $\Delta$ are not identified.
    \item If $\Delta$ has one boundary face, we
    preserve the 3-manifold if
    (i)~the one vertex of $\Delta$ not on this face is internal to the
    triangulation, and
    (ii)~no two of the three remaining (non-boundary) edges of $\Delta$
    are identified.
\end{itemize}

\subsubsection{Collapsing edges}

Our final move is the most powerful:
it collapses an edge of the triangulation to a single point,
and if the edge has high degree then it can eliminate many
tetrahedra.  The conditions for preserving the underlying 3-manifold are
complex, both to describe and to test algorithmically.  The details are
as follows.

\begin{defn}
    An \emph{edge collapse} operates on an edge $e$ of the triangulation
    that joins two distinct vertices.  It crushes edge $e$ to a point
    and flattens every tetrahedron containing $e$ to a face, as
    illustrated in Figure~\ref{fig-collapse}.

    \begin{figure}
    \centering
    \includegraphics[scale=0.7]{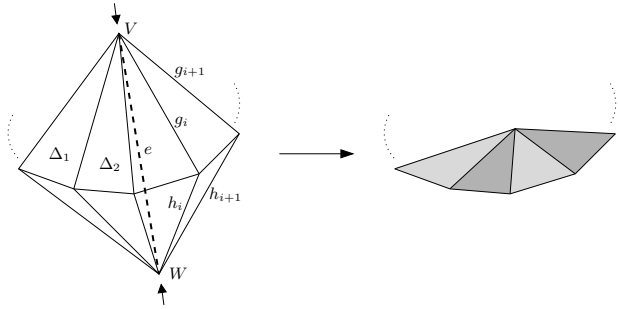}
    \caption{Collapsing an edge of a triangulation}
    \label{fig-collapse}
    \end{figure}
\end{defn}

Suppose the edge $e$ has degree $d$: we denote the $d$ tetrahedra that
contain it by $\Delta_1,\ldots,\Delta_d$, and we denote the two
endpoints of $e$ by $V$ and $W$.

If $e$ is an \emph{internal} edge of the triangulation
(i.e., its endpoints may lie on the boundary but its relative interior
must not), then the following conditions are sufficient to preserve the
underlying 3-manifold.
Section~\ref{s-simp-imp} gives details on how we can
test these conditions efficiently.

\begin{enumerate}
    \item The tetrahedra $\Delta_1,\ldots,\Delta_d$ must all be distinct.

    \item The two endpoints $V$ and $W$ (which we already know to be
    distinct) must not both be in the boundary.

    \item Denote the $d$ ``upper'' edges in the diagram that touch $V$
    by $g_1,\ldots,g_d$, and denote the corresponding ``lower'' edges
    that touch $W$ by $h_1,\ldots,h_d$
    (so each $g_i,h_i$ pair will be merged together by the edge collapse).
    We form a multigraph $\Gamma$
    (allowing loops and/or multiple edges) as follows:
    \begin{itemize}
        \item each distinct
        edge of the triangulation becomes a node of $\Gamma$;
        \item for each $i=1,\ldots,d$, we add an arc of
        $\Gamma$ between the two nodes corresponding to edges $g_i$ and $h_i$;
        \item we add an extra node $\partial$ to represent the boundary,
        and add an arc from $\partial$ to every node that represents a
        boundary edge.
    \end{itemize}
    Then this multigraph $\Gamma$ must not contain any cycles.%
    \label{en-bigons}

    \item In a similar way, we build a multigraph whose nodes represent
    \emph{faces} of the triangulation, whose arcs join corresponding
    ``upper'' and ``lower'' faces that touch $V$ and $W$ respectively,
    and with an extra boundary node connected to every boundary face.
    Again, this multigraph must not contain any cycles.%
    \label{en-pillows}
\end{enumerate}

In essence, condition~(\ref{en-bigons}) ensures that we never flatten a
chain of bigons whose outermost edges are both identified or both boundary,
and condition~(\ref{en-pillows}) ensures that we never flatten a chain of
triangular pillows whose outermost faces are both identified or both boundary.

If the edge $e$ lies in the boundary of the triangulation then we
may still be able to perform the move: the sufficient
conditions for preserving the 3-manifold are similar but slightly more
complex, and we refer the reader to {\regina}'s well-documented source
code for the details.

\subsection{The full simplification algorithm} \label{s-simp-alg}

Now that we are equipped with our suite of local simplification
moves, we can present the full details of {\regina}'s simplification
algorithm.  This algorithm is designed to be both fast and
effective, in that order of priority, and its underlying mechanics
have evolved over many years according to what has been found to work
well in practice.

Of course, other software packages---such as
{\snappea} \cite{snappea} and the {\recogniser}
\cite{recogniser}---have simplification algorithms of their own.
To date there has been no
comprehensive comparison between them; indeed, such a comparison would
be difficult given the ever-present trade-off between speed and
effectiveness.

\begin{algorithm} \label{a-simp}
Given an input triangulation $\tri$, the following procedure attempts to
reduce the number of tetrahedra in $\tri$ without changing the
underlying 3-manifold:
\begin{enumerate}
    \item Greedily reduce the number of tetrahedra as far as possible.
    We do this by repeatedly applying the following moves, in the
    following order of priority, until no such moves are possible:
    \begin{itemize}
        \item edge collapses;
        \item low-degree edge moves (3-2 Pachner moves,
            or 2-0 or 2-1 edge moves);
        \item low-degree vertex moves (2-0 vertex moves);
        \item boundary shelling moves.
    \end{itemize}%
    \label{en-simp-greedy}

    \item Make up to $5 R$ successive random 4-4 moves, where $R$ is the
    maximum number of available 4-4 moves that could be made from any
    single triangulation obtained during this particular iteration of
    step~(\ref{en-simp-44}).
    If we ever reach a triangulation from which we can greedily reduce
    the number of tetrahedra (as defined above) then return
    immediately to step~(\ref{en-simp-greedy}).%
    \label{en-simp-44}

    \item If the triangulation has boundary, then perform book opening
    moves until no more are possible.  If this enables us to collapse an
    edge then do so and return to step~(\ref{en-simp-greedy}).
    Otherwise undo the book openings and continue.%
    \label{en-simp-open}

    \item If a book closing move can be performed,
    then do it and return to step~(\ref{en-simp-greedy}).
    Otherwise terminate the algorithm.%
    \label{en-simp-close}
\end{enumerate}
\end{algorithm}

The greedy reduction in step~(\ref{en-simp-greedy})
prioritises edge collapses, because these can
remove many tetrahedra at once, and because we typically aim for
a one-vertex triangulation.
In step~(\ref{en-simp-44}), the coefficient $5$ is chosen
somewhat arbitrarily; note also that the quantity $R$ might increase
as this step progresses.
The book openings in step~(\ref{en-simp-open}) aim to increase
the number of vertices without adding new tetrahedra, in the hope
that an edge collapse becomes possible.
The book closures in step~(\ref{en-simp-close}) aim to leave us with the
smallest boundary possible, which becomes advantageous during other
expensive algorithms (such as normal surface enumeration).

Note that we never explicitly \emph{increase} the number of
tetrahedra (e.g., we never perform an explicit 2-3 Pachner move).
This is only possible because we have a large suite of moves
available: if we reformulate our algorithm in terms of Pachner
moves alone then almost every step would require both 2-3 and 3-2 moves.

One of the more prominent simplification techniques that we
do \emph{not} use is the 0-efficiency reduction of Jaco and
Rubinstein \cite{jaco03-0-efficiency}.  This is because the best
known algorithms for testing for 0-efficiency run in worst-case
exponential time.  Moreover, experimental observation suggests
that---ignoring well-known exceptions, such as reducible manifolds
and $\R P^3$---after running Algorithm~\ref{a-simp} we typically
find that the triangulation is already 0-efficient.
We discuss algorithms for 0-efficiency testing further in
Section~\ref{s-normal-eff}.

\subsection{Time complexity and performance} \label{s-simp-imp}

Simplification is one of the most commonly used ``large-scale''
routines in {\regina}'s codebase, and it is imperative that it runs
quickly.  In this section we analyse the running time, which includes a
discussion of key implementation details for the edge collapse.
All time complexities are based on the word RAM model of computation
(so, for instance, adding two integers is considered $O(1)$ time as long
as the integers only require $\log n$ bits).

Throughout this discussion, we let $n$ denote the number of tetrahedra
in the triangulation.  Some key points to note:
\begin{itemize}
    \item \emph{Adding a new tetrahedron} takes $O(1)$ time.
    However, \emph{deleting a tetrahedron} takes $O(n)$ time
    for {\regina} since we must reindex the
    tetrahedra that remain (i.e., if we delete tetrahedron $\Delta_i$
    then we must rename $\Delta_j$ to $\Delta_{j-1}$ for all $j>i$).\footnote{%
        This is {\regina}'s own implementation constraint:
        many moves could be $O(1)$ if we ignored {\regina}'s need to
        consecutively index tetrahedra and other objects.
        Either way, however, the time complexity for the
        edge collapse---and hence the full simplification algorithm---would
        be the same.}
    Deleting \emph{many} tetrahedra can be done in combined $O(n)$
    time, since if we are careful in our implementation then each
    leftover tetrahedron only needs to be reindexed at most once.
    \item \emph{Computing the skeleton} of the triangulation (i.e.,
    identifying and indexing the distinct vertices, edges, faces and boundary
    components of the triangulation, and linking these to and from the
    corresponding tetrahedra) can be done in
    $O(n)$ time using standard depth-first search techniques.
\end{itemize}

\begin{theorem}
    Algorithm~\ref{a-simp} (the full simplification algorithm)
    runs in $O(n^4 \log n)$ time, where $n$ is the number of tetrahedra
    in the input triangulation.
\end{theorem}

The worst culprit in raising the time complexity is the edge collapse
move, and specifically, testing its sufficient conditions.
In practice running times are much faster than quartic, and with more delicacy
we could bring down the theoretical time complexity to reflect this;
we return to such issues after the proof.

\begin{proof}
    First, we observe that for every move type except for the edge collapse,
    we can test
    the sufficient conditions in $O(1)$ time and perform the move in
    $O(n)$ time (where the dominating factor is deleting tetrahedra and
    rebuilding the skeleton, as outlined above).

    For the edge collapse move, performing the move takes $O(n)$ time
    but testing the sufficient conditions is a little slower.
    Recall that we must build a multigraph $\Gamma$ and ensure that it
    contains no cycles.  We can do this by adding one arc at a time, and
    tracking connected components: if an arc joins two distinct
    components then we merge them into a single component, and if an arc
    joins some component with itself then we obtain a cycle and the
    sufficient conditions fail.

    To track connected components, we use the well-known
    \emph{union-find} data structure \cite{cormen01-algorithms}.
    With union-find, the operations of (i)~identifying which component a
    node belongs to and (ii)~merging two components together each take
    $O(\log n)$ time.  Therefore testing the multigraph $\Gamma$ for
    cycles takes $O(n \log n)$ time overall, and testing sufficient
    conditions for an edge collapse likewise becomes $O(n \log n)$.

    From here the running time is simple to obtain.
    The algorithm works through ``stages'', where in each stage we either
    reduce the number of tetrahedra, or we reduce the number of boundary
    faces through a book closing move.  Either way, it is clear there
    can be at most $O(n)$ such stages in total (since the triangulation
    and its boundary cannot disappear entirely).

    Within each stage we \emph{perform} $O(n)$ moves
    in total: at most $5R \leq 5 (\textrm{\#\ edges}) \leq 5 \cdot 6n$
    successive 4-4 moves, and at most
    $(\textrm{\#\ internal\ faces}) \leq 2n$ successive book opening moves.
    This gives us a total of $O(n^2)$ moves throughout the life of the
    algorithm, each requiring $O(n)$ time to perform.

    Between each pair of moves, we might \emph{test} a large number of
    potential moves whose sufficient conditions ultimately fail.
    The number of moves that we test on a given triangulation is
    clearly $O(n)$ (since there are $O(n)$ possible ``local regions'' in
    which each type of move could be performed), and so throughout the
    entire algorithm with its $O(n^2)$ moves we run a total of $O(n^3)$
    tests.  In the worst case (edge collapses), each test could take
    $O(n \log n)$ time.

    It is clear now that the total time spent \emph{performing} moves is
    $O(n^3)$ and the total time spent \emph{testing} sufficient
    conditions is $O(n^4 \log n)$, yielding a running time of
    $O(n^4 \log n)$ overall.
\end{proof}

Some further remarks on this running time:
\begin{itemize}
    \item
    We noted earlier that \emph{in practice} running times are faster than
    $O(n^4 \log n)$.  This is because the powerful edge collapse
    moves typically reach the fewest possible number of vertices very
    quickly (i.e., one vertex for a closed manifold, or else one vertex for
    each boundary component).  Once we achieve this, testing sufficient
    conditions for an edge collapse move becomes $O(1)$ time, which
    eliminates an $n \log n$ factor from our running time.

    \item
    In theory, we can remove a factor of $n$ as follows:
    once greedy simplification fails and we
    move on to steps~(\ref{en-simp-44}) and~(\ref{en-simp-open}),
    we only test for \emph{new} simplification moves in the immediate
    neighbourhood of the \emph{last} 4-4 or book opening move.
    The implementation becomes more subtle and the bookkeeping more
    complex, and this is planned for future versions of {\regina}.

    \item
    Finally, we note that the $\log n$ factor can be stripped down to
    ``almost constant'': essentially an inverse of the Ackermann function,
    and $\leq 4$ in all conceivable situations.
    We can do this by applying the \emph{path compression}
    optimisation to the union-find data structure; see
    \cite{cormen01-algorithms,tarjan75-efficiency} for details.
\end{itemize}

We finish this section with a practical demonstration.
Let $\tri$ be the 23-tetra\-hedron triangulation of the
Weber-Seifert dodecahedral space presented in \cite{burton12-ws},
and let $S$ be an arbitrarily-chosen vertex normal surface of the
highest possible genus (here we choose vertex surface \#1733,
which is an orientable genus~16 surface).
If we cut $\tri$ open along the surface $S$
using {\regina}'s \texttt{cutAlong()} procedure,
we obtain a (disconnected) bounded triangulation $\tri'$ with 1990 tetrahedra.
The simplification algorithm reduces this to 135 tetrahedra in roughly
$1.9$~seconds (as measured on a 2.93\ GHz Intel Core i7).

If we perform a \emph{barycentric subdivision} on $\tri'$, we obtain a
new triangulation $\tri''$ with 47\,760 tetrahedra (and
% 9920 vertices,
a large number of vertices,
which means a large number of expensive edge collapse moves).
Again the simplification algorithm reduces this to 135 tetrahedra,
but this time takes 780~seconds, suggesting (for this arbitrary example)
only a quadratic---not quartic---growth rate in $n$.

\subsection{Exhaustive simplification via the Pachner graph} \label{s-simp-bfs}

We finish our discussion on simplification with a new technology:
\emph{breadth-first search through the Pachner graph}.
The key idea is, instead of using greedy simplification heuristics,
to try \emph{all possible sequences} of Pachner moves
(up to a user-specified limit).  The result is a much slower, but also
much stronger, simplification algorithm.  This algorithm is based on
ideas from the large-scale experimental study of Pachner graphs
described in \cite{burton11-orprime},
and the code will soon be merged into {\regina}'s main source tree.

\begin{defn}
    For a closed 3-manifold triangulation $\mfd$, the
    \emph{restricted Pachner graph} $\rpg$ is the infinite graph whose
    nodes correspond to isomorphism classes of one-vertex triangulations
    of $\mfd$ (where by \emph{isomorphism} we mean a relabelling of
    tetrahedra and/or their vertices),
    and where two nodes are joined by an arc if there is a 2-3 or 3-2
    Pachner move between the corresponding triangulations.
    The nodes of $\rpg$ are partitioned into finite \emph{levels}
    $1,2,\ldots$ according to the number of tetrahedra in the
    corresponding triangulations.
\end{defn}

The basic idea is as follows.  By a result of Matveev
\cite{matveev87-transformations,matveev03-algms}, if we exclude
level~1 then the restricted Pachner graph $\rpg$
is always connected, and so we should be able to simplify a non-minimal
triangulation of $\mfd$ by finding a path through $\rpg$ from the
corresponding node to some other node at a lower level.  In detail:

\begin{algorithm} \label{a-bfs}
    Given a one-vertex triangulation $\tri$ with $n$ tetrahedra
    representing a closed 3-manifold $\mfd$,
    as well as a user-defined ``height parameter'' $h$,
    the following procedure attempts to find a triangulation of $\mfd$
    with $<n$ tetrahedra:
    \begin{enumerate}
        \item Conduct a breadth-first search through $\rpg$
        starting from the node representing $\tri$, but restrict
        this search to consider only nodes at levels $\leq n+h$.

        \item If we ever reach a node at level $<n$, this yields a
        simpler triangulation (which we try to simplify further with
        the fast Algorithm~\ref{a-simp}).
        Otherwise we advise the user to try again with a
        larger $h$ (if they can afford to do so).
    \end{enumerate}
\end{algorithm}

The height parameter is needed because $\rpg$ is infinite, and because
even the individual levels grow extremely quickly (for $\mfd=S^3$ the growth
rate is at least exponential, and it is open as to whether
it is super-exponential \cite{benedetti11-spheres}).
Extremely small height parameters work very well in practice:
for $\mfd=S^3$ there are no known cases for which $h=2$ will not
suffice \cite{burton11-pachner}.

Because of the enormous number of nodes involved,
a careful implementation of Algorithm~\ref{a-bfs} is vital.
We cannot afford to build even an entire single level of $\rpg$ beforehand;
instead we construct nodes as reach them, and cache them using their
isomorphism signatures (polynomial-time computable strings that also
manage isomorphism testing \cite{burton11-pachner}).  The algorithm
lends itself well to \emph{parallelisation} (using multithreading with
shared memory, not large-scale distributed processing).

Unlike the earlier Algorithm~\ref{a-simp}, this new
Algorithm~\ref{a-bfs} is certainly not poly\-no\-mial-time.  Even if we are
able to simplify the triangulation using $k$ Pachner moves for small $k$,
we still test a total of $O((n+h)^k)$ potential moves; moreover,
the growth rate
of $k$ itself is unknown, and so even the exponent could become exponential.
See \cite{burton11-pachner} for experimental measurements of $k$ for a
range of different 3-manifolds.

In practice, with height parameter $h=2$, Algorithm~\ref{a-bfs} easily
simplifies the $26$ ``pathological'' triangulations of $S^3$ with $\leq 10$
tetrahedra that the faster Algorithm~\ref{a-simp} could not.
For these cases we need $5 \leq k \leq 9$ moves, with CPU times
ranging from $0.8$ to $14$ seconds.
This is quite slow for just $n \leq 10$ tetrahedra, and highlights the
ever-present trade-off between speed and effectiveness.

%------------------------------ Normal surfaces --------------------------

\section{Normal and almost normal surfaces}

One of {\regina}'s core strengths is its ability to enumerate and work
with normal and almost normal surfaces.  A \emph{normal surface} in
a 3-manifold triangulation $\mathcal{T}$ is a
properly embedded surface in $\mathcal{T}$ that meets each tetrahedron
in a (possibly empty) collection of disjoint curvilinear triangles and/or
quadrilaterals, as illustrated in Figure~\ref{fig-normaldiscs}.
An \emph{octagonal almost normal surface} is defined in the same way, but
also requires that exactly one tetrahedron contains exactly one additional
octagonal piece, as illustrated in Figure~\ref{fig-octagon}.

\begin{figure}
\centering
\subfigure[Triangles and quadrilaterals]{\label{fig-normaldiscs}%
    \qquad\quad\includegraphics[scale=0.5]{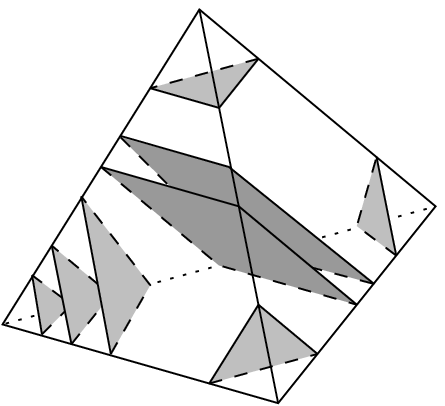}\qquad\quad}
\hspace{1cm}
\subfigure[An octagonal piece]{\label{fig-octagon}%
    \quad\includegraphics[scale=0.5]{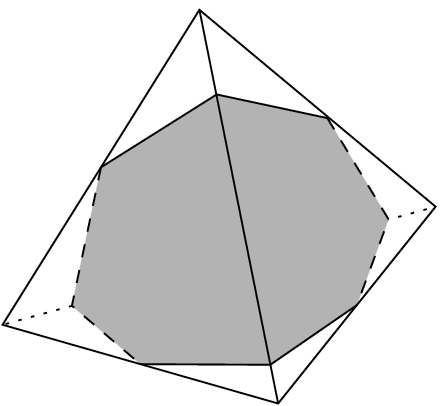}\quad}
\caption{A tetrahedron intersecting a normal or almost normal surface}
\end{figure}

Normal surfaces are a powerful tool for
high-level recognition and decomposition algorithms;
prominent examples include unknot recognition \cite{haken61-knot},
connected sum decomposition \cite{jaco03-0-efficiency},
and testing for incompressible surfaces \cite{jaco84-haken}.
Almost normal surfaces were introduced by Rubinstein,
and play a central role in algorithms such as 3-sphere recognition
\cite{rubinstein95-3sphere}, Heegaard genus \cite{lackenby08-tunnel},
and recognising small Seifert fibred spaces \cite{rubinstein04-smallsfs}.
Rubinstein originally defined almost normal surfaces to include either
a single octagonal piece or a single \emph{tube} piece, but Thompson later
showed that for 3-sphere recognition, only octagons need to be
considered \cite{thompson94-thinposition}.

We begin in Section~\ref{s-normal-prelim} with a
very brief overview of the necessary
concepts from normal surface theory; for more context
we refer the reader to \cite{hass99-knotnp}.

In Section~\ref{s-normal-enum} we discuss the all-important problem
of \emph{enumerating} vertex and fundamental normal surfaces,
and introduce a new trie-based optimisation to alleviate the most
severe bottlenecks in the enumeration algorithm.
We follow in Section~\ref{s-normal-eff} with a brief discussion of
0-efficiency and the important problem of locating normal \emph{spheres},
describing the rationale behind {\regina}'s choice of algorithm,
and explaining why other well-known options are not effective.

Section~\ref{s-normal-sphere}
outlines {\regina}'s current implementations of 3-sphere recognition, 3-ball
recognition and connected sum decomposition.  Although the key ideas
are already known,
the implementations have evolved to the point where all three
algorithms are now surprisingly simple, and we present them here as a useful
reference in a modern algorithmic form that is ``ready for implementation''.

We finish in Section~\ref{s-normal-tree} with a brief discussion of
\emph{tree traversal} algorithms, a new technology soon to appear
in {\regina} based on backtracking and linear
programming, and with enormous potential for improving
performance on large problems.

Beyond the algorithms described here, {\regina} offers many ways to
analyse normal surfaces, both ``at a glance'' and in detail.
It supports the complex operation of cutting a triangulation open along
a normal surface and retriangulating,
and it supports the Jaco-Rubinstein operation of
crushing a surface to a point \cite{jaco03-0-efficiency}
(which may introduce additional changes in topology).

\subsection{Preliminaries from normal surface theory} \label{s-normal-prelim}

In an $n$-tetrahedron triangulation $\tri$,
normal surfaces correspond to integer vectors in a cone of the form
$\{\mathbf{x} \in \R^{7n}\,|\,A\mathbf{x}=0,\ \mathbf{x}\geq 0\}$,
where the matrix $A$ of \emph{matching equations} is derived from $\tri$.
The $7n$ coordinates are grouped into $4n$ \emph{triangle coordinates},
which count the triangles at each corner of each tetrahedron,
and $3n$ \emph{quadrilateral coordinates}, which count the
quadrilaterals passing through each tetrahedron in each of the three
possible directions.
Such vectors must also satisfy the \emph{quadrilateral constraints},
which require that at most one quadrilateral coordinate within each
tetrahedron can be non-zero.
These constraints map out a (typically non-convex) union of faces of the
cone above.

An important observation is that non-trivial connected normal surfaces can be
reconstructed from their $3n$ quadrilateral coordinates alone
\cite{tollefson98-quadspace}.  We can therefore identify such
surfaces with integer points in a smaller-dimensional cone of the form
$\{\mathbf{x} \in \R^{3n}\,|\,B\mathbf{x}=0,\ \mathbf{x}\geq 0\}$.
We refer to $\R^{7n}$ and $\R^{3n}$ as working in
\emph{standard coordinates} and \emph{quadrilateral coordinates} respectively.

A normal surface is called a (standard or quadrilateral) \emph{vertex surface}
if its vector in (standard or quadrilateral) coordinates
lies on an extreme ray of the corresponding cone,
and it is called a (standard or quadrilateral) \emph{fundamental surface}
if its vector lies in the Hilbert basis of the cone.
The quadrilateral vertex and fundamental surfaces are typically a strict
subset of their standard counterparts.

Throughout this paper, we use the phrase \emph{almost normal surface} to
refer exclusively to the case where the extra piece is an octagon (not a
tube).  For almost normal surfaces we introduce three additional octagon
coordinates for each tetrahedron, yielding a cone in
\emph{standard almost normal coordinates} of the form
$\{\mathbf{x} \in \R^{10n}\,|\,C\mathbf{x}=0,\ \mathbf{x}\geq 0\}$.
As before, non-trivial connected surfaces can be reconstructed from their
$3n$ quadrilateral and $3n$ octagon coordinates \cite{burton10-quadoct},
yielding a cone in \emph{quadrilateral-octagon coordinates} of the form
$\{\mathbf{x} \in \R^{6n}\,|\,D\mathbf{x}=0,\ \mathbf{x}\geq 0\}$.
We can likewise define vertex and fundamental surfaces in these
coordinate systems.

To finish, we make the well-known observation that
Euler characteristic is a linear function in standard normal
and almost normal coordinates \cite{jaco95-algorithms-decomposition},
though it is not linear in quadrilateral or quadrilateral-octagon coordinates.

\subsection{Enumeration} \label{s-normal-enum}

Many high-level algorithms are based on locating particular surfaces,
which---if they exist---can be found as vertex normal surfaces,
or for some more difficult algorithms, fundamental normal surfaces.
{\regina} comes with heavily optimised algorithms for enumerating all
vertex normal surfaces \cite{burton10-dd} or fundamental normal surfaces
\cite{burton11-hilbert} in a triangulation, in all of the coordinate
systems listed above.

Here we focus on the vertex enumeration algorithm, which is based on the
double description method for enumerating extreme rays of polyhedral
cones \cite{fukuda96-doubledesc,motzkin53-dd}.  We outline the double
description method very briefly, and then introduce a new trie-based
optimisation that yields significant improvements in its running time.

In brief, the double description method enumerates the extreme rays of
the cone
$\{\mathbf{x} \in \R^{d}\,|\,A\mathbf{x}=0,\ \mathbf{x}\geq 0\}$
by constructing a series of cones $\mathcal{C}_0,\mathcal{C}_1,\ldots$,
where each $\mathcal{C}_i$ is defined only using the first
$i$ rows of $A$.  The initial cone $\mathcal{C}_0$ is simply the
non-negative orthant, with extreme rays defined by the $d$ unit vectors,
and each subsequent cone $\mathcal{C}_i$ is obtained inductively from
$\mathcal{C}_{i-1}$ by intersecting with a new hyperplane $H_i$.
The extreme rays of $\mathcal{C}_i$ are obtained from (i)~extreme rays
of $\mathcal{C}_{i-1}$ that lie on $H_i$; and
(ii)~convex combinations of \emph{pairs} of adjacent extreme rays of
$\mathcal{C}_{i-1}$ that lie on either side of $H_i$.

There are significant optimisations that can be applied to the
double description method in the context of normal surface theory; see
\cite{burton09-convert,burton10-dd} for details.  However, a major
problem remains: the intermediate cones $\mathcal{C}_i$ can have a great
many extreme rays (the well-known ``combinatorial explosion'' in the
double description method), and the combinatorial algorithm\footnote{%
    There is an alternative algebraic algorithm; see
    \cite{fukuda96-doubledesc} for theoretical and practical comparisons.}
for identifying all pairs of \emph{adjacent} extreme rays of $\mathcal{C}_i$
is cubic in the total number of extreme rays.

This cubic procedure stands out as the most severe bottleneck in the
algorithm.  The adjacency test is simple: two extreme rays
$\mathbf{x}_1,\mathbf{x}_2$ of $\mathcal{C}_i$
are adjacent if and only if there is no
other extreme ray $\mathbf{z}$ of $\mathcal{C}_i$ for which,
whenever the $i$th coordinates of $\mathbf{x}_1$ and $\mathbf{x}_2$
are both zero, then the $i$th coordinate of $\mathbf{z}$ is zero also.
The cubic algorithm essentially just iterates through all possibilities for
$\mathbf{x}_1$, $\mathbf{x}_2$ and $\mathbf{z}$.

We improve this procedure by storing the extreme rays of $\mathcal{C}_i$
in a \emph{trie} (also known as a \emph{radix tree}) \cite{sedgewick92}.
In our case this is a binary tree of depth $d$,
illustrated in Figure~\ref{fig-trie}, where at depth $i$ the left and
right branches contain all extreme rays for which the $(i+1)$th coordinate is
zero and non-zero respectively.\footnote{%
    This is well-defined because each extreme ray is of the form
    $\{\lambda \mathbf{v}\,|\,\lambda \geq 0\}$ for some vector
    $\mathbf{v}$.}
The extreme rays themselves correspond
to leaves of the tree at depth $d$.
We only store those portions of the tree that
contain extreme rays as descendants, so the total number of nodes is
$O(d \cdot \mbox{\# extreme rays})$ and not $2^{d+1}-1$.

\begin{figure}
    \centering
    \includegraphics[scale=0.9]{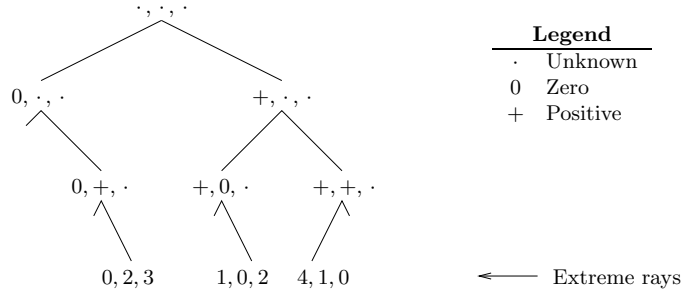}
    \caption{An example in $\R^3$ of storing extreme rays in a trie}
    \label{fig-trie}
\end{figure}

To identify all pairs of adjacent extreme rays of $\mathcal{C}_i$,
we first insert all extreme rays into the trie: each insertion takes $O(d)$
time (just follow a path down and create new nodes as needed).
We then iterate through all pairs of extreme rays $\mathbf{x}_1,\mathbf{x}_2$,
and to test \emph{adjacency}
we walk through the trie by (i)~starting at the root,
and (ii)~whenever we reach a node at level $i$, if the $(i+1)$th coordinates of
$\mathbf{x}_1$ and $\mathbf{x}_2$ are both zero then we follow the left
branch, and otherwise we follow both the left and right branches
in a depth-first manner.  We declare $\mathbf{x}_1$ and $\mathbf{x}_2$ to
be adjacent if and only if we do \emph{not} locate some other extreme
ray $\mathbf{z} \neq \mathbf{x}_1,\mathbf{x}_2$ during our walk.

We can optimise this trie further:
\begin{itemize}
    \item At each node, we store the number of extreme rays in the
    corresponding subtree.  This allows us to avoid the ``false
    positives'' $\mathbf{x}_1$ and $\mathbf{x}_2$: if we are in a
    subtree containing one or both of $\mathbf{x}_1,\mathbf{x}_2$ and the
    number of extreme rays is one or two respectively, then we can
    backtrack immediately.

    \item We can ``compress'' the trie by storing extreme rays at the
    nodes corresponding to their last \emph{non-zero} coordinate,
    instead of at depth $d$, a useful optimisation given that
    our extreme rays may contain many zeroes.
\end{itemize}

The appeal of this data structure is that we are able to target our
search by only looking at ``promising'' candidates for $\mathbf{z}$,
instead of scanning through all extreme rays.
It is difficult to pin down the theoretical complexity of our trie-based
search; certainly it might be exponential in $d$ (because of the
branching), but it is clearly no worse than
$O(d \cdot \mbox{\# extreme rays})$, i.e., the total number of nodes.

In practice, it serves us very well.
Consider again the 23-tetrahedron triangulation of the Weber-Seifert
dodecahedral space from \cite{burton12-ws}.
With the original implementation of the double description method
(including all optimisations except for the trie-based search),
enumerating all 698 quadrilateral vertex normal surfaces requires
$174$ minutes (measured on a 2.93\ GHz Intel Core i7).
The new trie-based algorithm reduces this to $62$ minutes,
cutting the running time from roughly three hours down to just one.

\subsection{0-efficiency} \label{s-normal-eff}

An important problem in normal surface theory is searching for
normal \emph{spheres}: this is at the heart of Jaco and Rubinstein's
0-efficiency machinery \cite{jaco03-0-efficiency}, and features in all
of the high-level algorithms listed in Section~\ref{s-normal-sphere}.

A closed orientable 3-manifold is \emph{0-efficient}
if its only normal 2-spheres are the trivial
vertex linking spheres (which contain only triangles, and which
must always be present).
If a triangulation is not 0-efficient, then we can use Jaco and Rubinstein's
``destructive crushing'' procedure \cite{jaco03-0-efficiency}:
we crush the non-trivial sphere to a point, and then collapse away
any degenerate non-tetrahedron pieces (such as footballs, pillows and so on)
to become edges and faces, as illustrated in Figure~\ref{fig-jrcrush}.

\begin{figure}
    \centering
    \includegraphics[scale=0.9]{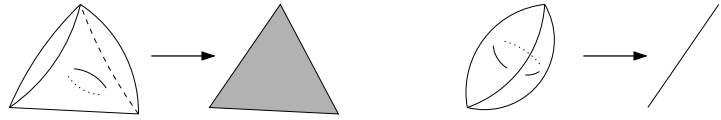}
    \caption{Examples of collapsing non-tetrahedron pieces}
    \label{fig-jrcrush}
\end{figure}

The result is that every tetrahedron that contains a quadrilateral
from the non-trivial sphere will disappear
entirely, and the final triangulation (which might be disconnected)
will be closed and have strictly fewer tetrahedra than the original.
The crushing process might introduce topological changes, but these
are limited to pulling apart connected sums, adding new 3-sphere components,
and deleting $S^3$, $\R P^3$, $S^2 \times S^1$ and/or $L_{3,1}$ components.
The crushing process is simple to implement (in stark contrast to the
messy procedure of \emph{cutting} along a normal surface),
and with help from first homology groups any topological changes
are easy to detect.

In order to test for 0-efficiency (and to explicitly identify a non-trivial
normal sphere if one exists), {\regina} uses the following result:

\begin{lemma}
    If a closed orientable triangulation $\tri$ contains a
    non-vertex-linking normal sphere, then it contains one as a
    quadrilateral vertex normal surface.
\end{lemma}

\begin{proof}
This result is widely known, but (to the author's best knowledge) does not
appear in the literature, and so we outline the simple proof here.
When we convert to vectors in \emph{standard} coordinates, any connected
non-vertex-linking normal surface $F$
can be expressed as a positive rational combination
of one or more \emph{quadrilateral} vertex normal surfaces,
minus zero or more vertex linking spheres.  Since the Euler characteristic
$\chi$ is linear in standard coordinates and $\chi(S^2) > 0$,
it follows that if $F$ is a non-trivial normal sphere then
some quadrilateral vertex normal surface $Q$ must have $\chi(Q) > 0$,
whereupon some rational multiple of $Q$ must be a non-trivial normal sphere
also.
\end{proof}

At present, {\regina} tests for 0-efficiency by enumerating
\emph{all} quadrilateral vertex normal surfaces (up to
multiples) and testing each.
This of course is more work than we need to do,
since we only need to locate \emph{one} non-trivial sphere.
We outline some tempting alternatives now, and explain why {\regina} does
not use them.

The first alternative is that,
in standard coordinates, we could restrict our polyhedral cone by
adding the homogeneous linear constraint $\chi \geq 0$.
This has been tried in {\regina}, but yields a substantially
\emph{slower} algorithm.  Experimentation suggests that this is because
(i)~we are forced to work in the higher-dimensional
$\R^{7n}$ instead of $\R^{3n}$; and
(ii)~the constraint $\chi \geq 0$ slices through the cone in a way
that creates significantly more extreme rays, exacerbating the
combinatorial explosion in the double description method.

The second alternative is based on linear programming.
Casson and Jaco et~al.\ \cite{jaco02-algorithms-essential}
have suggested (in essence) that, for each of the $3^n$ choices of
which quadrilateral coordinate we allow to be non-zero in each tetrahedron,
we could solve a linear program (in polynomial time) to
maximise $\chi$ over a corresponding sub-cone in $\R^{7n}$.
This is a promising approach, but it has a significant problem:
for ``good'' triangulations, which typically \emph{are} 0-efficient,
we must attempt all $3^n$ linear programs before we can terminate.
That is, $3^n$ becomes a \emph{lower bound} on the running time.

Although the best known theoretical time
complexity for a full vertex enumeration is slower than this
\cite{burton13-tree}, in \emph{practice} a full enumeration
is typically much faster.
For example, when enumerating quadrilateral vertex
normal surfaces for the first 1000 triangulations in the
Hodgson-Weeks closed hyperbolic census \cite{hodgson94-closedhypcensus}
(a good source of ``difficult'' manifolds for normal surface enumeration),
the optimised double description method gives a running time that
grows roughly like $1.6^n$, as shown in Figure~\ref{fig-dd-hyp}.

\begin{figure}
    \centering
    \includegraphics[scale=0.62]{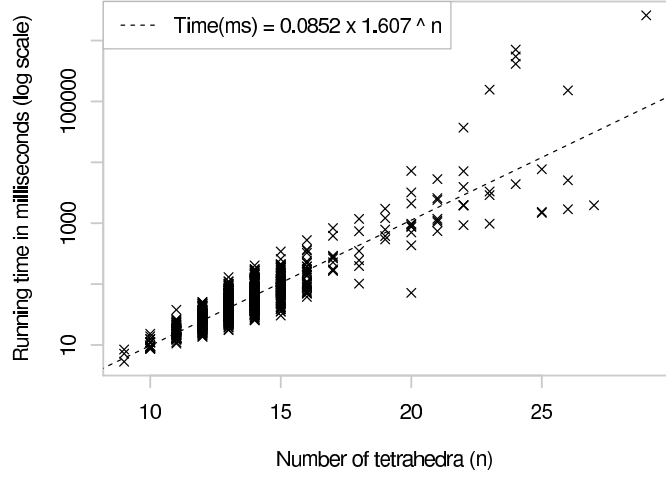}
    \caption{Growth of the running time for the double description method}
    \label{fig-dd-hyp}
\end{figure}
 
\subsection{High-level algorithms} \label{s-normal-sphere}

Here we present {\regina}'s current implementations of the high-level
3-sphere recognition, 3-ball recognition and connected sum decomposition
algorithms.  As noted earlier, the key ideas are already known; the
purpose of this description is to give a useful reference for these
algorithms in a modern ``ready to implement'' form.

It should be noted that all three algorithms guarantee both
correctness and termination (i.e., they are not probabilistic in nature).
The 3-sphere recognition and connected sum decomposition algorithms include
developments from many authors
\cite{burton10-quadoct,jaco89-surgery,jaco03-0-efficiency,jaco95-algorithms-decomposition,kneser29-normal,rubinstein95-3sphere,thompson94-thinposition},
and the 3-ball recognition algorithm is a trivial modification of 3-sphere
recognition.
See \cite{matveev03-algms} for a related but non-equivalent
variant of 3-sphere recognition based on special spines.

\begin{algorithm}[3-sphere recognition] \label{a-sphere}
    The following algorithm tests whether a given triangulation $\tri$
    is a triangulation of the 3-sphere.

    \begin{enumerate}
        \item Test whether $\tri$ is closed, connected and orientable.
        If $\tri$ fails any of these tests, terminate and return {\false}.
        \item Simplify $\tri$ using Algorithm~\ref{a-simp}.
        \item Test whether $\tri$ has trivial homology.
        If not, terminate and return {\false}.%
        \label{en-sphere-h1}
        \item Create a list $\mathcal{L}$ of triangulations to
        process, initially containing just $\tri$. \\
        While $\mathcal{L}$ is non-empty:
        \begin{itemize}
            \item Let $\mathcal{N}$ be the next triangulation in
            the list $\mathcal{L}$.  Remove $\mathcal{N}$ from
            $\mathcal{L}$, and test whether $\mathcal{N}$
            has a quadrilateral vertex normal sphere $F$.
            \begin{itemize}
                \item
                If so, then perform the Jaco-Rubinstein
                crushing procedure on $F$.
                For each connected component $\mathcal{N}'$
                of the resulting triangulation, simplify $\mathcal{N}'$
                and add it back into the list $\mathcal{L}$.
                \item
                If not, and if $\mathcal{N}$ has only one vertex,
                then search for a quad\-ri\-la\-teral-octagon vertex
                almost normal sphere in $\mathcal{N}$.
                If none exists then terminate and return \false.
            \end{itemize}
        \end{itemize}
        \item
        Once there are no more triangulations in $\mathcal{L}$,
        terminate and return {\true}.
    \end{enumerate}
\end{algorithm}

The key invariant in the algorithm above is that
the original 3-manifold is always the connected sum of all manifolds in
$\mathcal{L}$.
The homology test in step~(\ref{en-sphere-h1}) is crucial, since
the Jaco-Rubinstein crushing procedure could silently delete $S^2 \times S^1$,
$\R P^3$ and/or $L_{3,1}$ components.

\begin{algorithm}[3-ball recognition] \label{a-ball}
    The following algorithm tests whether a given triangulation $\tri$
    is a triangulation of the 3-ball.
    \begin{enumerate}
        \item Test whether $\tri$ is connected, orientable,
        has precisely one boundary component, and this
        boundary component is a 2-sphere.
        If $\tri$ fails any of these tests, terminate and return {\false}.
        \item Simplify $\tri$ using Algorithm~\ref{a-simp}.
        \item Cone the boundary of $\tri$ to a point by
        attaching one new tetrahedron to each boundary face,
        and simplify again.
        \item
        Run 3-sphere recognition over the final triangulation, and
        return the result.
    \end{enumerate}
\end{algorithm}

For our final algorithm, we note that by ``connected sum decomposition''
we mean a decomposition into \emph{non-trivial} prime summands
(i.e., no unwanted $S^3$ terms).

\begin{algorithm}[Connected sum decomposition] \label{a-connsum}
    The following algorithm computes the connected sum decomposition of
    the manifold described by a given triangulation $\tri$.
    We assume as a precondition that $\tri$ is closed, connected
    and orientable.
    \begin{enumerate}
        \item Simplify $\tri$ using Algorithm~\ref{a-simp}.
        \item Compute the first homology of $\tri$, and let
        $r$, $t_2$ and $t_3$ denote the rank,
        $\Z_2$~rank and $\Z_3$~rank respectively.
        \item Create an input list $\mathcal{L}$ of triangulations to
        process, initially containing just $\tri$, and an output list
        $\mathcal{O}$ of prime summands, initially empty. \\
        While $\mathcal{L}$ is non-empty:
        \begin{itemize}
            \item Let $\mathcal{N}$ be the next triangulation in
            the list $\mathcal{L}$.  Remove $\mathcal{N}$ from
            $\mathcal{L}$, and test whether $\mathcal{N}$
            has a quadrilateral vertex normal sphere $F$.
            \begin{itemize}
                \item
                If so, then perform the Jaco-Rubinstein
                crushing procedure on $F$.
                For each connected component $\mathcal{N}'$
                of the resulting triangulation, simplify $\mathcal{N}'$
                and add it back into the list $\mathcal{L}$.
                \item
                If not, then
                append $\mathcal{N}$ to the output list $\mathcal{O}$
                if either (i)~$\mathcal{N}$ has non-trivial homology,
                or (ii)~$\mathcal{N}$ has only one vertex
                and no quad\-ri\-la\-teral-octagon vertex
                almost normal sphere.
            \end{itemize}
        \end{itemize}
        \item Compute the first homology of each triangulation in
        the output list $\mathcal{O}$, sum the ranks,
        $\Z_2$~ranks and $\Z_3$~ranks, and append additional
        copies of $S^2 \times S^1$, $\R P^3$ and $L_{3,1}$
        to $\mathcal{O}$ so that these ranks sum to
        $r$, $t_2$ and $t_3$ respectively.
    \end{enumerate}
    On termination, the output list $\mathcal{O}$ will contain
    triangulations of the (non-trivial) prime summands of the input manifold.
\end{algorithm}

The key invariants of this algorithm are that (i)~the input manifold is always
the connected sum of all manifolds in $\mathcal{L}$ and $\mathcal{O}$,
plus zero or more $S^2 \times S^1$, $\R P^3$ and/or $L_{3,1}$ summands;
and that
(ii)~every output manifold in $\mathcal{O}$ is prime and not $S^3$.

\subsection{Tree traversal algorithms} \label{s-normal-tree}

There have been recent interesting developments in computational normal
surface theory that could allow us to move away from the double
description method entirely.  These are algorithms based on traversing a
search tree \cite{burton13-tree,burton12-unknot},
and they combine aspects of
linear programming, polytope theory and data structures.

The resulting algorithms avoid the dreaded combinatorial explosion of
the double description method; moreover, they offer incremental
output and are well-suited to parallelisation, progress tracking and
early termination.
Most importantly, experimentation suggests that they are significantly
faster and less memory-hungry---even when run in serial---for larger
and more difficult problems.
The code is already up and running, and will be included in the next
release of {\regina}.

Such tree traversal algorithms can be used for either a full enumeration
of vertex normal surfaces \cite{burton13-tree}, or to locate a single
non-trivial normal or almost normal sphere
(for 0-efficiency testing and/or 3-sphere recognition)
\cite{burton12-unknot}.
The key idea is to build a search tree according to which quadrilateral
coordinates are non-zero in each tetrahedron, and to run incremental
linear programs that enforce the quadrilateral constraints for those
tetrahedra where decisions have been made, but \emph{ignore} the
quadrilateral constraints for those tetrahedra that we have not yet
processed.

For the full enumeration of vertex normal surfaces, details of the
tree traversal algorithm can be found in \cite{burton13-tree}.
To illustrate, we return again to the 23-tetrahedron triangulation
of the Weber-Seifert dodecahedral space: whereas the trie-based double
description method enumerates all 698 quadrilateral vertex normal
surfaces in 62 minutes, the tree traversal algorithm does
this in just 32~minutes.

For locating just a single normal or almost normal sphere,
the tree traversal algorithm becomes extremely powerful:
it can prove that this same triangulation
of the Weber-Seifert dodecahedral space is 0-efficient in
\emph{under 10~seconds}.  Note that there is no early termination
here---the tree traversal algorithm conclusively proves in under 10~seconds
that no non-trivial normal sphere exists.

This latter algorithm for locating normal and almost normal spheres
relies on a number of crucial heuristics, and full details can be found
in \cite{burton12-unknot}.
Perhaps most interesting is the following experimental observation:
for ``typical'' inputs, this algorithm appears to require only a
\emph{linear} number of linear programs; that is, the typical behaviour
appears to be \emph{polynomial-time}.
One should be quick to note that this is in experimentation only,
and that the algorithm is not polynomial-time
in the worst case.  Nevertheless,
this is a very exciting computational development.

%------------------------------ Recognition ------------------------------

\section{Combinatorial recognition} \label{s-rec}

In this brief section we outline {\regina}'s combinatorial recognition code.
Despite significant
advances in 3-manifold algorithms, we as a community are still a
long way from being able to implement the full homeomorphism
algorithm---even simpler problems such as JSJ decomposition have
never been implemented, and Hakenness testing (which plays a key role in
the homeomorphism problem) has only recently become practical
\cite{burton12-ws}.
These are the issues that we aim to address (or rather work around) here.

In addition to slower but always-correct and always-conclusive
algorithms such as 3-sphere recognition and connected sum decomposition,
{\regina} offers a secondary means
for identifying 3-manifolds:
\emph{combinatorial recognition}.  The central idea is that
we ``hard-code'' a large number of general constructions
for infinite families of 3-manifolds (such as Seifert fibred spaces,
surface bundles and graph manifolds).
Then, given an input triangulation $\tri$, we test whether
$\tri$ follows one of these hard-coded constructions, and if it does,
we ``read off'' the parameters to name the underlying 3-manifold.

The advantages of this technique are:
\begin{itemize}
    \item It is extremely fast---all of
    {\regina}'s hard-coded constructions
    of infinite families can be recognised in small polynomial time.
    \item It allows {\regina} to recognise a much
    larger range of 3-manifolds than would otherwise be practically
    possible.
\end{itemize}

There are, of course, clear disadvantages:
\begin{itemize}
    \item Such techniques require a \emph{lot} of code if we wish to
    recognise each construction in its full generality:
    {\regina} currently has over $25\,000$ lines of source code
    devoted to combinatorial recognition alone.
    % As of 11 August 2012, the figure is 26596.
    \item The code is only as powerful and general as the constructions
    that are implemented.  For instance, with a handful of exceptions,
    {\regina}'s recognition routines do not include any hyperbolic manifolds.
    \item To be recognised, a triangulation must
    be \emph{well-structured}---an arbitrary triangulation of
    even a simple manifold such as a lens space
    will not be recognised unless it follows one of the known constructions.
\end{itemize}

Despite these drawbacks, combinatorial recognition is enormously
useful in practice.  {\regina}'s recognition code is particularly strong
for non-orientable manifolds: of the 366 manifolds in the
$\leq 11$-tetrahedron closed non-orientable census \cite{burton11-genus},
{\regina} is able to recognise \emph{all} minimal triangulations for
334 of them, and is able to recognise
at least \emph{one} minimal triangulation
for 357---that is, all but nine.

If {\regina} cannot recognise the manifold from an input
triangulation $\tri$ (i.e., the combinatorial recognition is inconclusive),
then often a good strategy is to modify $\tri$ so that the triangulation
becomes more ``well-structured''.  This might include (i)~simplifying
$\tri$ using Algorithm~\ref{a-simp}, or (ii)~the slower but stronger technique
of performing a breadth-first from $\tri$ through Pachner graph until we
reach a triangulation that \emph{can} be recognised, following the discussion
in Section~\ref{s-simp-bfs}.

{\regina} is not the only software package to
employ combinatorial recognition: there is also the {\recogniser}
by Matveev et~al.,
which has extremely powerful recognition heuristics that can recognise
a much wider range of 3-manifolds than {\regina} can.
See \cite{matveev03-algms,matveev12-owr,recogniser} for details.

%------------------------------ Angle structures -------------------------

\section{Angle structures}

In addition to normal surfaces, {\regina} can also enumerate and analyse
\emph{angle structures} on a triangulation $\tri$.
An angle structure assigns non-negative internal dihedral angles
to each edge of each tetrahedron of $\mathcal{T}$, so that
(i)~opposite edges of a tetrahedron are assigned the same angle;
(ii)~all angles in a tetrahedron sum to $2\pi$; and
(iii)~all angles around any internal edge of $\mathcal{T}$
likewise sum to $2\pi$ (see Figure~\ref{fig-angles}).
Such structures are often called
\emph{semi-angle structures} \cite{kang05-taut2}, to distinguish them from
\emph{strict angle structures} in which all angles are strictly positive.
Note that, by a simple Euler characteristic computation,
an angle structure can only exist if $\tri$ is an ideal triangulation
with every vertex link a torus or Klein bottle.

\begin{figure}
    \centering
    \includegraphics{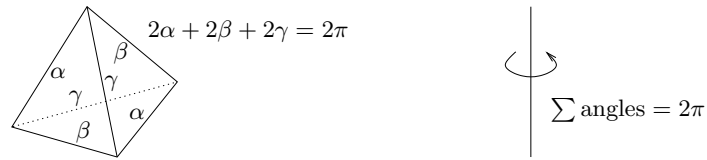}
    \caption{The conditions for an angle structure on a triangulation}
    \label{fig-angles}
\end{figure}

Angle structures were introduced by Rivin
\cite{rivin94-structures,rivin03-combopt} and Casson,
with further development by Lackenby \cite{lackenby00-anglestruct},
and are a simpler (but weaker) combinatorial analogue
of a complete hyperbolic structure.
Some angle structures are of particular interest:
these include \emph{taut angle structures}\footnote{%
    We follow the nomenclature of Hodgson et~al.\
    \cite{hodgson11-veering}---these are slightly more general
    than the original taut structures of Lackenby \cite{lackenby00-taut},
    who also adds a coorientation constraint.}
in which
every angle is precisely $0$ or $\pi$ (representing ``flattened'' tetrahedra)
\cite{hodgson11-veering,lackenby00-taut}, and
\emph{veering structures} which are taut angle structures with
powerful combinatorial constraints \cite{agol11-ideal,hodgson11-veering}.
All of these objects have an interesting role to play in
building a complete hyperbolic structure on
$\mathcal{T}$ \cite{futer11-angled,hodgson11-veering,kang05-taut2}.

Conditions (i)--(iii) above map out a polytope in $\R^{3n}$,
where $n$ is the number of tetrahedra; the vertices of this polytope
are called \emph{vertex angle structures}, and their convex combinations
generate all possible angle structures on $\mathcal{T}$.
For many years now, {\regina} has been able to enumerate all
vertex angle structures using the double
description method, as outlined in Section~\ref{s-normal-enum}.
Moreover, it can detect taut angle structures and (more recently) veering
structures when they are present.

A newer development is that {\regina} can enumerate \emph{only}
taut angle structures.  Detecting even a single taut angle structure
is NP-complete \cite{burton13-taut}; nevertheless, {\regina} can
enumerate all taut angle structures for relatively large
triangulations---in ad-hoc experiments it can do this for
70--80 tetrahedra in a matter of minutes.

The underlying algorithm is based on the following simple observation:

\begin{lemma} \label{l-taut-vertex}
    Every taut angle structure is also a vertex angle structure.
\end{lemma}

\begin{proof}
    Describing angle structures by vectors in $\R^{3n}$ as outlined above,
    suppose that $\tau = \lambda \alpha_1 + (1-\lambda)\alpha_2$
    where $\lambda \in (0,1)$, and where $\tau,\alpha_1,\alpha_2$
    are angle structures with $\tau$ taut.
    Then both $\alpha_1$ and $\alpha_2$ must have dihedral angles of
    zero wherever $\tau$ has a dihedral angle of zero, whereupon
    it follows that $\alpha_1 = \alpha_2 = \tau$.
\end{proof}

\begin{algorithm}
    Given an $n$-tetrahedron triangulation $\tri$,
    the following algorithm enumerates all taut angle structures
    on $\tri$.

    First, we projectivise the polytope described by conditions (i)--(iii)
    above.  This is a standard construction:
    we add a $(3n+1)$th coordinate, embed the original polytope
    $\mathcal{P}$ in the hyperplane $x_{3n+1}=1$, and build the
    \emph{cone} from the origin through $\mathcal{P}$.
    This replaces our bounded polytope $\mathcal{P} \subseteq \R^{3n}$
    with a polyhedral cone $\mathcal{C} \subseteq \R^{3n+1}$ of the form
    $\mathcal{C} =
        \{\mathbf{x} \in \R^{3n+1}\,|\,A\mathbf{x}=0,\ \mathbf{x}\geq 0\}$.
    The vertex angle structures in the original polytope
    $\mathcal{P}$ now correspond to the extreme rays of the cone $\mathcal{C}$.

    We run the double description method to enumerate all extreme
    rays of $\mathcal{C}$.
    Recall that this inductively constructs cones
    $\mathcal{C}_0,\mathcal{C}_1,\ldots$, where $\mathcal{C}_i$ is
    obtained from $\mathcal{C}_{i-1}$ by intersecting with a new
    hyperplane $H_i$, and where each extreme ray of $\mathcal{C}_i$
    is either (a)~an extreme ray of $\mathcal{C}_{i-1}$ that lies on $H_i$,
    or (b)~the convex combination of two adjacent extreme rays
    $\mathbf{x}_1,\mathbf{x}_2$ of $\mathcal{C}_{i-1}$ that lie
    on opposite sides of $H_i$.

    Here we introduce a simple but crucial optimisation:
    in case~(b) above, we only consider pairs $\mathbf{x}_1,\mathbf{x}_2$
    that \emph{together} do not have positive values in more than one
    coordinate position per tetrahedron.
\end{algorithm}

This optimisation is similar to Letscher's filtering method for
normal surface enumeration \cite{burton10-dd}.  It works because,
if some pair $\mathbf{x}_1,\mathbf{x}_2$
fails the final condition above, then
(by virtue of the fact that all angles are non-negative and we always
perform convex combinations) any vertex angle structure that we
eventually obtain from a combination of $\mathbf{x}_1$ and $\mathbf{x}_2$
must have multiple non-zero coordinates in some tetrahedron,
and so cannot be taut.

Finally, we note that we can improve the enumeration algorithm further,
both for enumerating taut angle structures and
enumerating \emph{all} vertex angle structures,
by employing the same trie-based optimisation that
we introduce in Section~\ref{s-normal-enum}.

%------------------------------ Experimentation --------------------------

\section{Experimentation} \label{s-exp}

In this penultimate section we illustrate how {\regina} can be used for
both small-scale and large-scale experimentation, in the hope that
readers can use this as a template for beginning their own experiments.

In addition to its graphical user interface,
{\regina} offers a powerful scripting facility, in which most of the
{\cpp} classes and functions in its mathematical engine are made available
through a dedicated {\python} module.
{\python} is a popular scripting language that is easy to write and easy
to read, and the {\python} module in {\regina} makes it easy to quickly
prototype new algorithms, run tests over large bodies of census data,
or perform complex tasks that would be cumbersome through a
point-and-click interface.

Users can access {\regina}'s {\python} module in two ways:
\begin{itemize}
    \item by opening a {\python} console from within the graphical user
    interface, which allows users to study or modify data
    in the current working file;

    \item by starting the command-line program \texttt{regina-python},
    which brings up a stand\-alone {\python} prompt.
\end{itemize}

Users can also run their own {\python} scripts directly via
\texttt{regina-python}, embed scripts within data files as
\emph{script packets},
or write their own libraries of frequently-used routines that
will be loaded automatically each time a {\regina} {\python} session
starts.

The following sample {\python} session constructs the triangulation
of $\R P^3$ that was illustrated in Section~\ref{s-simp-tri},
prints its first homology group,
enumerates all vertex normal surfaces, and then locates and prints the
coordinates of the vectors that represent vertex normal projective planes.

{\small \begin{verbatim}
bab@rosemary:~$ regina-python 
Regina 4.93
Software for 3-manifold topology and normal surface theory
Copyright (c) 1999-2012, The Regina development team
>>> tri = NTriangulation()
>>> t0 = tri.newTetrahedron()
>>> t1 = tri.newTetrahedron()
>>> t0.joinTo(0, t1, NPerm4(1,0,3,2))       # Glues 0 (123) -> 1 (032)
>>> t0.joinTo(1, t1, NPerm4(1,0,3,2))       # Glues 0 (023) -> 1 (132)
>>> t0.joinTo(2, t1, NPerm4(0,1,2,3))       # Glues 0 (013) -> 1 (013)
>>> t0.joinTo(3, t1, NPerm4(0,1,2,3))       # Glues 0 (012) -> 1 (012)
>>> print tri.getHomologyH1()
Z_2
>>> s = NNormalSurfaceList.enumerate(tri, NNormalSurfaceList.STANDARD, 1)
>>> print s
5 vertex normal surfaces (Standard normal (tri-quad))
>>> for i in range(s.getNumberOfSurfaces()):
...     if s.getSurface(i).getEulerCharacteristic() == 1:
...         print s.getSurface(i)
... 
0 0 0 0 ; 0 1 0 || 0 0 0 0 ; 0 1 0
0 0 0 0 ; 0 0 1 || 0 0 0 0 ; 0 0 1
>>> 
\end{verbatim}}

One of {\regina}'s most useful facilities for experimentation is its
ability to create \emph{census data}: exhaustive lists of all 3-manifold
triangulations (up to combinatorial isomorphism) that satisfy some given
set of constraints.  The census algorithms are heavily optimised
\cite{burton04-facegraphs,burton07-nor10,burton11-genus},
and can be run in serial on a desktop or in parallel on a
large supercomputer.

The simplest way for users to create their own census data is through
the command-line \texttt{tricensus} tool.  The following example
constructs all 532 closed orientable 3-manifold triangulations with
$n=4$ tetrahedra:
{\small \begin{verbatim}
bab@rosemary:~$ tricensus --tetrahedra=4 --internal --orientable --finite
                --sigs output.txt
Starting census generation...
0:1 0:0 1:0 1:1 | 0:2 0:3 2:0 2:1 | 1:2 1:3 3:0 3:1 | 2:2 2:3 3:3 3:2
0:1 0:0 1:0 1:1 | 0:2 0:3 2:0 3:0 | 1:2 2:2 2:1 3:1 | 1:3 2:3 3:3 3:2
0:1 0:0 1:0 1:1 | 0:2 0:3 2:0 3:0 | 1:2 3:1 3:2 3:3 | 1:3 2:1 2:2 2:3
0:1 0:0 1:0 2:0 | 0:2 1:2 1:1 3:0 | 0:3 2:2 2:1 3:1 | 1:3 2:3 3:3 3:2
0:1 0:0 1:0 2:0 | 0:2 1:2 1:1 3:0 | 0:3 3:1 3:2 3:3 | 1:3 2:1 2:2 2:3
0:1 0:0 1:0 2:0 | 0:2 2:1 2:2 3:0 | 0:3 1:1 1:2 3:1 | 1:3 2:3 3:3 3:2
0:1 0:0 1:0 2:0 | 0:2 2:1 3:0 3:1 | 0:3 1:1 3:2 3:3 | 1:2 1:3 2:2 2:3
1:0 1:1 1:2 2:0 | 0:0 0:1 0:2 3:0 | 0:3 3:1 3:2 3:3 | 1:3 2:1 2:2 2:3
1:0 1:1 2:0 2:1 | 0:0 0:1 3:0 3:1 | 0:2 0:3 3:2 3:3 | 1:2 1:3 2:2 2:3
1:0 1:1 2:0 3:0 | 0:0 0:1 2:1 3:1 | 0:2 1:2 3:2 3:3 | 0:3 1:3 2:2 2:3
Finished.
Total triangulations: 532
\end{verbatim}}

The triangulations are converted to text-based isomorphism signatures
\cite{burton11-orprime} and written to the text file
\texttt{output.txt}, one per line.

We can now (as an illustration)
search through this census for all two-vertex 0-efficient
triangulations of the 3-sphere:

{\small \begin{verbatim}
bab@rosemary:~$ regina-python 
Regina 4.93
Software for 3-manifold topology and normal surface theory
Copyright (c) 1999-2012, The Regina development team
>>> f = open('output.txt', 'r')
>>> sig = f.readline()
>>> while sig:
...     sig = sig[0:-1]    # Strip off trailing newline
...     tri = NTriangulation.fromIsoSig(sig)
...     if tri.getNumberOfVertices() == 2:
...         if tri.isZeroEfficient() and tri.isThreeSphere():
...             print sig
...     sig = f.readline()
... 
eLAkaccddjgjqc
>>>
\end{verbatim}}

Here we see that, for $n=4$ tetrahedra,
there is one and only one such triangulation.
To study it in more detail, we can open up {\regina}'s
graphical user interface and create a new triangulation from the
isomorphism signature \texttt{eLAkaccddjgjqc}.

These small examples illustrate how, using the census facility and {\python}
scripting combined, {\regina} can be an invaluable tool for
experimentation, testing conjectures, and searching for pathological
examples.

%------------------------------ Future -----------------------------------

\section{The future of Regina} \label{s-future}

{\regina} continues to enjoy active development and regular releases.
The developers are currently working towards a major version 5.0 release,
which will also work with triangulated \emph{4-manifolds}
and normal hypersurfaces (in joint work with Ryan Budney).
Other coming developments include richer operations on triangulated
2-manifolds, visualisation of vertex links in 3-manifold triangulations
(by Budney and Samuel Churchill), and much more sophisticated algebraic
machinery in 2, 3 and 4 dimensions (by Budney).
Much of this code is already running and well-tested.

Users are encouraged to contribute code and offer feedback.
For information on new releases, interested parties are welcome to
subscribe to the low-traffic mailing list
\texttt{regina-announce@lists.sourceforge.net}.

%---------------------------- Bibliography -------------------------------

\bibliographystyle{amsplain}
\bibliography{pure}

\end{document}